\documentclass[reqno,centertags,12pt]{amsart}
\usepackage[letterpaper,margin=1.3in]{geometry}
\usepackage{amsmath,amsthm,amsfonts,amssymb,enumerate}
\usepackage[bookmarksopen=true,final]{hyperref}
\usepackage{multirow,yhmath}
\usepackage{xcolor}
\usepackage[utf8]{inputenc}
\usepackage{bm}
\usepackage{caption}

\usepackage{graphicx}

\newcommand{\bb}{\mathbb}

\newcommand{\supp}{\mathrm{supp}}

\renewcommand{\Re}{\mathrm{Re}}
\renewcommand{\Im}{\mathrm{Im}}

\newtheorem{theorem}{Theorem}
\newtheorem{lemma}[theorem]{Lemma}

\theoremstyle{definition}

\theoremstyle{remark}
\newtheorem*{remark}{Remark}
\numberwithin{equation}{section}
\numberwithin{theorem}{section}
\theoremstyle{remark}
\newtheorem*{remarks}{Remarks}

\newcommand{\bbR}{\mathbb{R}}
\newcommand{\bbC}{\mathbb{C}}
\newcommand{\bbH}{\mathbb{H}}
\newcommand{\calJ}{\mathcal{J}}
\newcommand{\Arg}{\operatorname{Arg}}

\begin{document}
	\title[Multiplicative perturbations of classical $\beta$-ensembles]{Multiplicative non-Hermitian perturbations of classical $\beta$-ensembles}	

	\author{G\"{o}kalp Alpan}
	\address{Faculty of engineering and natural sciences, Sabancı University, İstanbul,\newline Turkey.}
	\email{gokalp.alpan@sabanciuniv.edu}
	\author{Rostyslav Kozhan}
\address{Department of Mathematics, Uppsala University, Uppsala, Sweden.}
\email{rostyslav.kozhan@math.uu.se}

	\begin{abstract}
We compute the joint eigenvalue distribution for a 
 multiplicative non-Hermitian perturbation $(I+i\Gamma)H$, $\operatorname{rank}\,\Gamma=1$ of a
random matrix $H$ from the Gaussian, Laguerre, and chiral Gaussian $\beta$-ensembles.
	\end{abstract}
	
	\date{\today}
	\maketitle
	
	\section{Introduction}\label{ss:Intro}
	Let $H$ be an $n\times n$ random matrix from one of the 
	classical Hermitian ensembles of random matrices. 
	There is a vast amount of research in the mathematical and physical literature on additive non-Hermitian perturbations of $H$ of the form
	\begin{equation}\label{eq:additive}
	H+i\Gamma,
	\end{equation}
	where $\Gamma^*=\Gamma$, $\operatorname{rank}\,\Gamma \ll \operatorname{rank}\,H$.
	Such matrices have received considerable attention in recent years, particularly due to their applications in physics, see~\cite{fyo16,fyosav15,fyosom03,nucl} and references therein. In particular, the exact eigenvalue distribution of low rank non-Hermitian  perturbations of classical random matrix ensembles has been the subject of studies of \cite{AlpRos,fyokho99,Koz17,Koz20,SokZel,StoSeb,Ull} in particular; see also~\cite{ForrRev} for a nice review and~\cite{DubLas} for a related study of the dynamical problem.
			
	Instead of additive perturbations, O'Rourke--Wood in~\cite{OroWoo1} initiated a  study of {\it{multiplicative}} non-Hermitian perturbations
	\begin{equation}\label{eq:mult}
		(I_n+i\Gamma)H,
	\end{equation}
	where $I_n$ is the identity matrix, $\Gamma^*=\Gamma$, and $\operatorname{rank}\,\Gamma \ll \operatorname{rank}\,H$. 
 		
	The main result of this paper is the computation of the joint eigenvalue density for~\eqref{eq:mult} in the case where $\operatorname{rank}\,\Gamma=1$ and $H$ being from the Gaussian (orthogonal, unitary, or symplectic), Laguerre  (orthogonal, unitary, or symplectic), or chiral Gaussian (orthogonal, unitary, or symplectic) ensembles. 
	This is addressed in Sections~\ref{ss:GOE},~\ref{ss:LOE}, \ref{ss:GSE},~\ref{ss:chiral} below.
	
	More generally we compute the joint eigenvalue density for analogous multiplicative perturbations of the corresponding $\beta$-ensembles for any $\beta>0$, see Theorems~\ref{Gauss}, 
	\ref{thm:Lag1}, and \ref{laguerredifficult}. 

	For the proofs, we use the Dumitriu--Edelman tridiagonalization procedure ~\cite{Dumede02}, which allows us to employ the tools of orthogonal polynomials with random coefficients~\cite{AlpRos,KK,Koz17,Koz20}. 
	It is not clear to us if the explicit eigenvalue density for the additive~\eqref{eq:additive} or multiplicative~\eqref{eq:mult} perturbations in the case $\operatorname{rank}\,\Gamma\ge 2$ can be expressed in term of elementary functions. The orthogonal polynomial approach becomes inherently matrix-valued and technically taxing even in the case $\operatorname{rank}\,\Gamma = 2$. This still remains an interesting open question.
	
	A non-Hermitian random matrix of the form~\eqref{eq:additive} serves as a natural model for the energy Hamiltonian of an open quantum system. Mathematically, under suitable scaling, the eigenvalues of such matrices exhibit properties that are often referred to as ``weakly non-Hermitian'', as coined in~\cite{FKS,FKS2}. The asymptotic analysis of such matrices has been a popular topic in the mathematics and physics literature (see among many others, ~\cite{fyosav15,FyoSom96,FyoSom97,fyosom03,fyosomtit,nucl,OroWoo1} and references therein).

	
	
	
	
	\medskip
	
	\textbf{Acknowledgments}: The authors are grateful to the anonymous referees whose suggestions helped to improve the paper. Research of G. A. was supported by the BAGEP Award of the Science Academy.  The Research of G. A. was supported by the Vergstiftelsen Foundation during his postdoc at Uppsala University, where a significant portion of the work done.

	\section{Gaussian ensembles}
	\subsection{Gaussian $\beta$-ensemble and the spectral measure}\label{ss:GaussianSpectral}

	Given any $\beta>0$, we say that a tridiagonal (Jacobi) matrix
	\begin{equation}\label{genericJwithn}
		\calJ = \left(
		\begin{array}{ccccc}
			b_1&a_1&0& &\\
			a_1&b_2&a_2&\ddots &\\
			0&a_2&b_3&\ddots & 0 \\
			&\ddots&\ddots&\ddots & a_{n-1} \\
			& & 0 & a_{n-1} & b_n
		\end{array}\right), \quad a_j > 0, \quad b_n\in\bbR,
	\end{equation}
	belongs to the Gaussian $\beta$-ensemble  $\mathrm{G\beta E}_n$ if the joint distribution of the Jacobi coefficients $a_j$'s and $b_j$'s is proportional to 
	\begin{equation}\label{eq:GbetaCoefficients}
	 \prod_{j=1}^{n-1} a_j^{\beta(n-j)-1} \exp\left(- \sum_{j=1}^{n-1}a_j^2 - \tfrac12 \sum_{j=1}^n b_j^2 \right).
	\end{equation}
	Such matrices were introduced in~\cite{Dumede02} as the tridiagonal realization of the Gaussian orthogonal, unitary, and symplectic ensembles (see Sections~\ref{ss:GOE} and~\ref{ss:GSE} for further details). 
	They also computed the distribution of the spectral measure $\mu$, which is defined as the unique probability measure satisfying
	\begin{equation}\label{eq:spWithJ}
		\langle {e}_1, \calJ^k  {e}_1 \rangle = \int_\bbR x^k d\mu(x), \quad \mbox{for all } k\ge 0.
	\end{equation}
	It is a well-known fact that the spectral measure uniquely determines Jacobi coefficients and vice versa.
	For  $\mathrm{G\beta E}_n$,  the spectral measure is~\cite{Dumede02}
	\begin{equation}\label{spectralM}
		\mu(x) = \sum_{j=1}^n w_j \delta_{\lambda_j} , 
	\end{equation}
	with the joint distribution
	\begin{equation}\label{evGaussian}
		\tfrac{1}{g_{\beta,n}} \prod_{j=1}^n e^{-\lambda_j^2/2} \prod_{1\le j<k \le n} |\lambda_j - \lambda_k|^\beta d\lambda_1 \ldots d\lambda_n
		\times \tfrac{1}{c_{\beta,n}} \prod_{j=1}^n w_j^{\beta /2 -1 }dw_1 \ldots dw_{n-1},
	\end{equation}
	where
	\begin{align}
		\label{domain}
		& \sum_{j=1}^n w_j =1;  \quad w_j > 0, \quad 1 \le j \le n; \quad \lambda_j \in \bbR, \\
		\label{normalizations}
		& g_{\beta,n} = (2\pi)^{n/2} \prod_{j=1}^{n} \frac{\Gamma(1+\beta j/2)}{\Gamma(1+\beta/2)},
		\quad
		c_{\beta,n} = \frac{\Gamma(\beta/2)^n}{\Gamma(\beta n/2)}
	\end{align} 
and $\lambda_j$'s are the eigenvalues of the matrix $\calJ$ in \eqref{genericJwithn}.
		\subsection{Multiplicative perturbation of $\mathrm{G\beta E}_n$ }\label{ss:multGaus}
	
	The main object of study of the current paper is the multiplicative perturbation of $\calJ$ of the form
	\begin{equation}\label{genericJ2new}
		\calJ_{l,\times}:=(I_n+ile_1e_1^*)\mathcal{J}= 
		\left(
		\begin{array}{ccccc}
			b_1 +ilb_1&a_1+ila_1&0& &\\
			a_1&b_2&a_2&\ddots &\\
			0&a_2&b_3&\ddots & 0 \\
			&\ddots&\ddots&\ddots & a_{n-1} \\
			& & 0 & a_{n-1} & b_n
		\end{array}\right).
	\end{equation}
	
	Let us introduce the notation
	\begin{equation*}\label{arg.mod.pi}
		\Arg_{[0,\pi)} z :=
		\begin{cases}
			\Arg\,z\qquad&\text{if}\quad z\in\bbC_+,\\
			0\qquad&\text{if}\quad z\in\bbR,\\
			\pi+\Arg\,z\,\quad&\text{if}\quad z\in\bbC_-,\\
		\end{cases}
	\end{equation*}
	where $\Arg\,z \in (-\pi,\pi]$ is the principal value of the argument of $z$, $\bbC_+ = \{z\in\bbC:\Im\, z>0\}$, $\bbC_- = \{z\in\bbC:\Im \,z<0\}$. Geometrically $\Arg_{[0,\pi)} z \in [0,\pi)$ measures the angle between the radius-vector of $z$  and the positive  half-axis $\bbR_+$ if $z\in\bbC_+$ and  the negative half-axis
	$\bbR_-$ if $z\in\bbC_-$. Algebraically one can see that
	\begin{equation}\label{eq:cotan}
		\Arg_{[0,\pi)} z
		= \operatorname{arccot} \frac{\operatorname{Re} z}{\operatorname{Im} z} \in (0,\pi)
	\end{equation}
	for $z\in\bbC\setminus\bbR$.

	We are heading towards the proof of the following result. Its proof is given in Section~\ref{ss:Gaussian}.

	\begin{theorem}\label{Gauss}
		Let $\beta>0$, $\calJ$ belong to the Gaussian $\beta$-ensemble, 
		$l>0$ be independent of $\calJ$ with a given absolutely continuous distribution $d\nu (l)=F(l)\,dl$. 
		Then the eigenvalues $\{z_j\}_{j=1}^n$  of $\calJ_{l,\times}$ ~\eqref{genericJ2new} are distributed on 
		\begin{equation}\label{eq:configurations}
			\left\{ \{z_j\}_{j=1}^n \in (\bbC\setminus\bbR)^n: \tan\Big(\sum_{j=1}^n \Arg_{[0,\pi)} z_j \Big) \in \supp\,\nu
			\right\}.
		\end{equation}
		according to
		\begin{align}
			\frac{1}{C_{\beta,n}} & e^{-\frac{1}{2}((\sum_{j=1}^n \Re\, z_j)^2-\sum_{j\neq k} \Re(z_j z_k))}\nonumber\\
			&\times  \frac{1}{|\Re(\prod_{j=1}^n z_j)|^{\beta/2}}\prod_{j,k=1}^n |z_j-\bar{z_k}|^{\beta/2-1}\prod_{1\leq j<k\leq n}^n |z_j-{z_k}|^2\nonumber\\
			&\times  l^{1-\beta n/2} 
			F(l)
			 \, d^2 z_1\ldots d^2 z_n,\label{huge1}
		\end{align}
		where
		$l=\tan\Big(\sum_{j=1}^n \Arg_{[0,\pi)} z_j\Big)$ and $C_{\beta,n}=g_{\beta,n}c_{\beta,n} 2^{n(\beta/2-1)}$. 
	\end{theorem}
	\begin{remarks}
		1. For comparison, the joint distribution of rank-one {\it additive} perturbations $\calJ+ i l e_1 e_1^*$ is given by
		\begin{multline}\label{zFinal_random}
			\tfrac{1}{h_{\beta,n}} \, 
			e^{-\frac{1}{2}((\sum_{j=1}^n \Re\, z_j)^2-\sum_{j\neq k} \Re(z_j z_k))}
			 \prod_{j,k=1}^n |z_j-\bar{z}_k |^{\beta/2 -1} \prod_{j<k} |z_j-z_k|^2
			\\
			\times l^{1-{\beta n}/{2}} F(l) d^2 z_1\ldots  d^2 z_n,
		\end{multline}
		on
		\begin{equation*}
			\left\{ \{z_j\}_{j=1}^n \in (\bbC_+)^n: \sum_{j=1}^n \Im\, z_j  \in \supp\,\nu
			\right\},
		\end{equation*}
		where $l=\sum_{j=1}^n \Im\, z_j$, see  ~\cite[Thm 3]{Koz17} and ~\cite[Eq. (5.20)]{Koz17} for more details. The difference is the extra factor involving the product of $z_j$'s, another expression for $l$, and the supports of the distributions. See also the remark after Theorem~\ref{thm:Lag1} for comparison for the Laguerre case.
		
		2. By taking $F(l)dl$ to a Dirac delta at some $l_0>0$, one can then obtain the joint distribution of the eigenvalues for the case when $l$ is deterministic. In particular, the eigenvalues would be supported on the set where $\tan\big(\sum_{j=1}^n \Arg_{[0,\pi)} z_j \big) = l_0$.

		3. See Fig. 1 for an example of one realization of eigenvalues of $\calJ_{l,\times}$. Note that all the points are in the first and the third quadrants, and the sum of the angles between the radius-vector of eigenvalues and the $x$-axis (with its positive direction for the first quadrant and negative for the third quadrant)  is equal to $\arctan l =\pi/4$.
		
		4. When $n=1$, the problem in Theorem~\ref{Gauss} is reduced to the computation of the distribution of the complex random variable $z:=b(1+il)$, where $b$ is a  real normal variable $N(0,1)$ and $l$ has p.d.f. $F(l)$. Using \eqref{magni1}, the distribution ~\eqref{huge1} then simplifies (where $z=x+iy$) to
		$ e^{-x^2/2}  \frac{1}{|x|}
		F\left( \frac{|y|}{|x|}\right)\,dx\,dy$, up to a normalization.
	\end{remarks}
	
	\begin{figure}
		\centering
		\includegraphics[width=\linewidth]{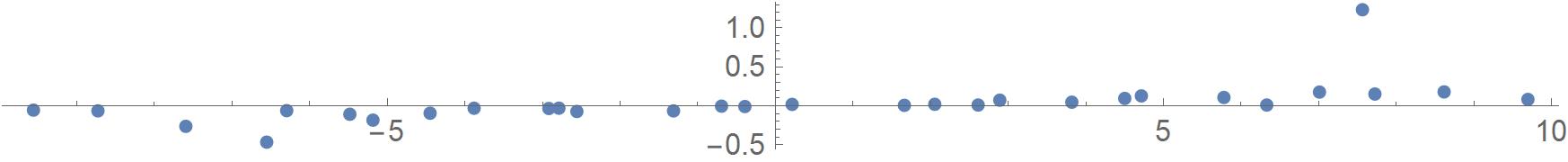}
		\captionsetup{labelformat=empty}
		\label{pic}\captionsetup{labelformat=empty}
		\caption{\tiny{Fig. 1. One realization of eigenvalues for $\calJ_{l,\times}$ ($n=30$, $\beta=2$, $l=1$) (via Wolfram Mathematica) }}
	\end{figure}

	\subsection{Eigenvalue configurations}
	
	Let us discuss the possible eigenvalue configurations of matrices of the form~\eqref{genericJ2new}.
	Let $\mathcal{J}^{(1)}$ be the Jacobi matrix obtained by deleting the first column and first row of $\mathcal{J}$, and let $\mathcal{J}^{(2)}$ be  the Jacobi matrix obtained by deleting the first two columns and first two rows of $\mathcal{J}$.
	
	\begin{lemma}\label{lem:charact}
		Let $p(z)$ and $q(z)$ be the characteristic polynomials of $\calJ$ and $\calJ^{(1)}$, respectively. Then 
		\begin{equation}\label{eq:HB}
			\det(zI_n -\calJ_{l,\times})=(1+il) p(z) - il z q(z).
		\end{equation}
	\end{lemma}
	\begin{proof}
		Let 
		$r(z)= \det(zI_{n-2}-\mathcal{J}^{(2)})$. 
		Expanding this determinant along the first row, we get
		\begin{equation}\label{largeeq}
			\det(zI_n -\calJ_{l,\times})= (z-b_1-ilb_1)q(z)-a_1^2r(z)-ila_1^2r(z)
		\end{equation}
		Taking $l=0$ we get
		$p(z)=(z-b_1)q(z)-a_1^2r(z)$ (this is the well-known three-term recurrence for monic orthogonal polynomials, see e.g., \cite{Sim11}). 
		This allows to express $r$ in terms of $p$ and $q$. Substituting it into \eqref{largeeq}, we get ~\eqref{eq:HB}.
	\end{proof}

	All possible eigenvalue configurations of matrices of the form $\calJ_{l,\times}$ (or $H_{l,\times}$, assuming $e_1$ is cyclic, see Section \ref{ss:GOE}) are described in Lemma~\ref{thm:HB} below. Roughly speaking~\cite{KT}, the effect of a rank 1 multiplicative perturbation is that  positive eigenvalues of $\calJ$ move into the interior of the first quadrant of $\bbC$, while negative eigenvalues move into the interior of the third quadrant. A zero eigenvalue of $\calJ$ leads to a zero eigenvalue of $\calJ_{l,\times}$ (this can be easily seen from ~\eqref{eq:HB}). This is why $\det\calJ=0$ becomes a special case. 
	
	For us, this will only be important when we study the Laguerre ensembles, for which $\calJ$ is either strictly positive definite or  positive semi-definite with one zero eigenvalue. 
	That is why in (c) we only consider the case $\calJ\ge 0$. If one removes this condition then~\eqref{eq:zeros3} would simply get $(\bbC_+)^{N-1}$ replaced by $(\bbC\setminus\bbR)^{N-1}$, but we do not need this case in the sequel.
	
	\begin{lemma}[{\cite[Cor. 5.1 and 5.2]{KT}}]\label{thm:HB}
		Fix $l>0$. Let $\mathcal{J}$ be an $N\times N$ Jacobi matrix as in \eqref{genericJwithn}. 
		\begin{itemize}
			\item [(a)] If $\det\mathcal{J}\ne 0$ then eigenvalues of $\calJ_{l,\times}$,~\eqref{genericJ2new}, belong to
			\begin{equation}\label{eq:zeros}
				\left\{ \{z_j\}_{j=1}^N \in (\bbC\setminus\bbR)^N: \sum_{j=1}^N \Arg_{[0,\pi)} z_j =\arctan l
				\right\}.
			\end{equation}
			Moreover, there is one-to-one correspondence between every configuration of points in~\eqref{eq:zeros} and a choice of $\calJ$ with $\det\mathcal{J}\ne 0$.
			
			The number of eigenvalues of $\calJ_{l,\times}$ in $\bbC_+$ (respectively, in $\bbC_-$) coincides with the number of positive (respectively, negative) eigenvalues of $\calJ$.
			\item[(b)] If $\mathcal{J}>0$ then eigenvalues of $\calJ_{l,\times}$,~\eqref{genericJ2new}, belong to
			\begin{equation}\label{eq:zeros2}
				\left\{ \{z_j\}_{j=1}^N \in (\bbC_+)^N: \sum_{j=1}^N \Arg\, z_j =\arctan l
				\right\}.
			\end{equation}
			Moreover, there is one-to-one correspondence between every configuration of points in~\eqref{eq:zeros2} and a choice of $\calJ$ with $\mathcal{J}>0$.
			
			\item[(c)] If $\mathcal{J}\ge 0$, 
			$\det \mathcal{J}=0$, then $\calJ_{l,\times}$,~\eqref{genericJ2new}, has a simple eigenvalue at $0$ and the remaining eigenvalues belong to
			\begin{equation}\label{eq:zeros3}
				\left\{ \{z_j\}_{j=1}^{N-1} \in (\bbC_+)^{N-1}: \sum_{j=1}^{N-1} \Arg\, z_j < \arctan l
				\right\}.
			\end{equation}
			Moreover, there is one-to-one correspondence between every configuration of points in \eqref{eq:zeros3} and a choice of $\calJ$ with $\mathcal{J}\ge 0$, $\det \mathcal{J}=0$. 
		\end{itemize}
	\end{lemma}
	
	\subsection{Jacobians}
	For a function $f: {\bb R}^s\rightarrow  {\bb R}^t$, its Jacobian matrix will be denoted by 
	\begin{equation*}
		\frac{\partial(f_1, \ldots f_t)}{\partial(x_1,\ldots, x_s)}.
	\end{equation*}
	
	Recall the map from $\mu$ of the form~\eqref{spectralM} with~\eqref{domain} to $\calJ$ of the form~\eqref{genericJwithn} is one-to-one and onto. There is another one-to-one correspondence that is of interest to us: namely, given $p(z)$ and $q(z)$ with interlacing zeros there is a unique $\calJ$ (and therefore $\mu$) with $p$ and $q$ being the characteristic polynomials of $\calJ$ and $\calJ^{(1)}$~\cite{Wen}.  Notice that if $\lambda_j$'s are known then $p(z)=\prod (z-\lambda_j)$ is also determined and therefore we can consider the map from $w_j$'s to the coefficients of $q(z)$ as one-to-one. The following technical result from \cite{Koz17} will prove to be useful.
	\begin{lemma}\label{lm:Jacq}
		Let $q(z)= \sum_{j=0}^{n-1} q_j z^j$ with $q_{n-1}=1$. Treating $\lambda_j$ in~\eqref{spectralM} as constants,
		\begin{equation}\label{eq:Jacq}
			\left\lvert \det{\frac{\partial{(q_0,\ldots, q_{n-2})}}{\partial{(w_1, \ldots, w_{n-1})}}} \right\rvert= \prod_{1\leq j<k\leq n}  |\lambda_j-\lambda_k|.
		\end{equation}
	\end{lemma}
	\begin{proof}
		This is shown in the proof of~\cite[Lemma 6]{Koz17}. Indeed, our $q(z)$ is in their notation $-\frac1l \sum_{k=0}^{n-1} (\Im\,\kappa_j) z^j$, that is, our $q_j$ is their $-\Im\,\kappa_j/l$ for each $j=0,\ldots,n-2$, where $l$ is some given constant. Therefore our Jacobian coincides with \cite[Eq.(5.14)]{Koz17} but without  the $l^{n-1}$ factor which is equal to the right-hand side of~\eqref{eq:Jacq}. 
	\end{proof}
	\begin{remark}
		The Jacobian~\eqref{eq:Jacq} has  been implicitly present in many other computations. First, note that it is trivially equivalent to 
		\begin{equation}\label{eq:newJac}
			\left\lvert \det{\frac{\partial{(\mu_0,\ldots, \mu_{n-2})}}{\partial{(w_1, \ldots, w_{n-1})}}} \right\rvert= \frac{\prod_{1\leq j<k\leq n} |\lambda_j-\lambda_k|}{\prod_{1\leq j<k\leq n-1} |\mu_j-\mu_k|} ,
		\end{equation}
		where $\mu_j$'s are the zeros of $q(z)$. This is exactly the same Jacobian (see~\cite[Eq~(4.14)]{Forrester-book}, where it is presented in a slightly more convoluted way) that appears in the Dixon--Anderson derivation~\cite{And,Dix} of the Selberg integral. Let us briefly summarize the idea behind the use of \eqref{eq:newJac}.	
		Given a sequence of reals  $\{\lambda_j\}_{j=1}^n$, we form the  random rational function $m(x):=\sum_{j=1}^n \frac{w_j}{\lambda_j - x}$ with the joint distribution of $w_j$'s on the simplex $w_j>0$, $\sum_{j=1}^n w_j = 1$ being proportional to
		$$
		\prod_{j=1}^n w_j^{s_j-1}, 
		$$
		for some constants $s_j>0$. Representing $m(x) = - {q(x)}/{p(x)}$, we can see that the roots of $m$ coincide with the roots of $q(x) = \prod_{j=1}^{n-1} (x-\mu_j)$, whose joint distribution is then easily computable using~\eqref{eq:newJac}. 
		The new distribution is a probability measure and hence has a total mass 1. This produces an identity which is just one step away from the recurrence formula for the Selberg integral, see~\cite[Sect~4.2]{Forrester-book} for further details. 
	\end{remark}
	
	In the next result, let us assume $l>0$ to be a random variable independent of $\mu$ (i.e., of $\calJ$). Then (see Lemma~\ref{thm:HB}(a), (b)) the mapping from $\{\mu,l\}$ to $\{z_j\}_{j=1}^n$ (eigenvalues of $\calJ_{l,\times}$) is one-to-one on the set $\prod_{j=1}^n \lambda_j \ne 0$. In the next lemma, we compute the Jacobian of this map.
	\begin{lemma}\label{lem:Jacobian}
		Let $\mathcal{J}$ be an $n\times n$ Jacobi matrix as in \eqref{genericJwithn} with $\det\calJ\ne0$ and the spectral measure given by \eqref{spectralM} with \eqref{domain}. Denote $\{z_j\}_{j=1}^n$ to be the eigenvalues of $\calJ_{l,\times}$,~\eqref{genericJ2new}, with $l>0$.  
		Then
		\begin{equation}\label{Jaco1}
			\left\lvert \det{\frac{\partial{(\Re\, z_1,\Im\,z_1,\ldots, \Re\,z_{n},\Im\,z_n)}}{\partial{(\lambda_1,\ldots, \lambda_{n},w_1, \ldots, w_{n-1},l)}}} \right\rvert= l^{n-1}\prod_{1\leq j<k\leq n}  \frac{|\lambda_j-\lambda_k|^2}{|z_j-z_k|^2} \Big\lvert\prod_{j=1}^n \lambda_j\Big\rvert.
		\end{equation}
	\end{lemma}

		\begin{proof}
			Denote  $\kappa_j$'s to be the coefficients of the characteristic polynomial of $\calJ_{l,\times}$:
			\begin{equation}
				\det(z I_n-\calJ_{l,\times})=\prod_{j=1}^n (z-z_j)= \sum_{j=0}^n \kappa_j z^j
			\end{equation}
			with $\kappa_n=1$.
			First of all, it is standard (see, e.g.,~\cite[Appendix D]{KK}) that
			\begin{equation}\label{vande0}
				\left\lvert \det{\frac{\partial{(\Re\,\kappa_0,\ldots,\Re\,\kappa_{n-1},\Im\,\kappa_{0},\ldots,\Im\,\kappa_{n-1})}}{\partial{(\Re\, z_1,\ldots,\Re\, z_{n},\Im\, z_{1},\ldots,\Im\, z_{n})}}}\right\rvert= \prod_{1\leq j<k\leq n}|z_j-z_k|^2.
			\end{equation}	
			Now we compute the Jacobian of the transformation from $\kappa_j$'s to $\lambda_j$'s, $w_j$'s, and $l$.
			
			Recall \eqref{eq:HB}:
			\begin{equation}\label{rew10}
				\prod_{j=1}^n(z-z_j)=\sum_{j=0}^n \kappa_j z^j=p(z)+ilp(z)-ilzq(z).
			\end{equation}
			It follows that
			\begin{align}
				& \label{Rekj0}
				\sum_{j=0}^{n}(\Re\, \kappa_j)z^j=\prod_{j=1}^n (z-\lambda_j)= p(z), \\
				& \label{rew20}
				\sum_{j=0}^{n-1} (\Im \,\kappa_j) z^j=lp(z)-lzq(z).
			\end{align}
			Then~\eqref{Rekj0} gives us (see, e.g.,~\cite[Appendix D]{KK}) 
			\begin{equation}\label{Jacob110}
				\left\lvert \det{\frac{\partial{(\Re\,\kappa_0,\ldots, \Re\,\kappa_{n-1})}}{\partial{(\lambda_1,\ldots, \lambda_{n})}}} \right\rvert= \prod_{1\leq j<k\leq n}  |\lambda_j-\lambda_k|. 
			\end{equation}
			Note that $\Re \,\kappa_j$'s do not depend on $w_j$'s or $l$ in \eqref{Rekj0}. Thus
			\begin{equation}\label{0mat10}
				{\frac{\partial{(\Re\,\kappa_0,\ldots, \Re\,\kappa_{n-1})}}{\partial{(w_1,\ldots, w_{n-1},l)} }}=[0]_{n\times n}.
			\end{equation}
			Substituting $z=0$ into \eqref{rew20}, we see that $\Im\, \kappa_0=l \Re\, \kappa_0 = l (-1)^n \prod_{j=1}^n \lambda_j$ and thus
			\begin{align}
				&\frac{\partial{\Im \,\kappa_0}}{\partial{w_j}}=0,\,\,\, \mbox{for } j=1,\ldots, n-1,\label{part1eq0}\\
				&\left\lvert\frac{\partial{\Im \,\kappa_0}}{\partial{l}}\right\rvert =|\Re\, \kappa_0|=\Big\lvert\prod_{j=1}^n \lambda_j\Big\rvert.\label{part2eq0}
			\end{align}
			Since the Jacobian in \eqref{vande0} can be written in block form and one block vanishes in view of \eqref{0mat10}, we are left with computing
			\begin{equation}\label{Jacobdet110}
				\left\lvert \det{\frac{\partial{(\Im\,\kappa_1,\ldots, \Im\,\kappa_{n-1})}}{\partial{(w_1,\ldots, w_{n-1})}}} \right\rvert,
			\end{equation}
			while treating $\lambda_j$'s and $l$ as constant. Let $q(z)= \sum_{j=0}^{n-1} q_j z^j$ with $q_{n-1}=1$ as in Lemma~\ref{lm:Jacq}.
			
			Note that \eqref{rew20} shows that  $\Im\,\kappa_j$ for $1\le j \le n-1$ is equal to $-l q_{j-1}$ plus terms independent of $w_j$'s. Using Lemma~\ref{lm:Jacq} we obtain that~\eqref{Jacobdet110} is equal to
			$$
			l^{n-1} \prod_{1\leq j<k\leq n}  |\lambda_j-\lambda_k|.
			$$
			
			Combining this with
			\eqref{vande0}, \eqref{Jacob110}, \eqref{0mat10}, \eqref{part1eq0}, \eqref{part2eq0}, we obtain~\eqref{Jaco1}.	
									\end{proof}
									
									\subsection{Proof of Theorem~\ref{Gauss}}\label{ss:Gaussian}
									\begin{proof}[Proof of Theorem~\ref{Gauss}]
										
												By \eqref{evGaussian} the joint density of ${w_1},\ldots,w_{n-1},\lambda_1,\ldots \lambda_n,l$ is
												\begin{equation}\label{Jaco27}
													\frac{F(l)}{g_{\beta,n}c_{\beta,n}}\prod_{j=1}^n e^{-\lambda_j^2/2}\prod_{j<k}  |\lambda_j-\lambda_k|^\beta \prod_{j=1}^n  w_j^{\beta/2-1}.
												\end{equation}
												All $\lambda_j$'s are distinct since they are eigenvalues of a Jacobi matrix $\calJ$. For the rest of the proof, we work conditionally on the event that $\det\calJ \ne 0$ which happens with probability 1. 
												By Lemma~\ref{thm:HB}(a), there is a bijective map from $\mu$ and $l$ to $\{z_j\}_{j=1}^n$ in~\eqref{eq:zeros}. Its  Jacobian is given in \eqref{Jaco1}.
												This implies that the induced  joint density of $z_1,\ldots,z_n$ is
												\begin{equation}\label{Jaco28}
													\frac{l^{1-n}F(l)}{\left\lvert\prod_{j=1}^n \lambda_j\right\rvert g_{\beta,n}c_{\beta,n}} \prod_{1\leq j<k\leq n}|z_j-z_k|^2\prod_{j=1}^n e^{-\lambda_j^2/2}\prod_{j<k}  |\lambda_j-\lambda_k|^{\beta-2} \prod_{j=1}^n  w_j^{\beta/2-1}.
												\end{equation}
												We now need to express \eqref{Jaco28} in terms of $z_j$'s.
												The equalities
												\begin{align}
													&l=\frac{\Im(\prod_{j=1}^n z_j)}{{\Re(\prod_{j=1}^n z_j)}} = 
													\tan\Big(\sum_{j=1}^n \Arg_{[0,\pi)} z_j\Big)
													,\label{magni1}\\
													&\Big\vert \Re \prod_{j=1}^n z_j \Big\vert =	\Big\vert \prod_{j=1}^n \lambda_j \Big\vert \label{magni2}
												\end{align}
												follow from \eqref{rew10} at $z=0$ (the right-hand side of~\eqref{magni1} is a consequence of \eqref{eq:cotan} and~\eqref{eq:zeros}).
												Using the coefficients of the terms $z^{n-1}$ and $z^{n-2}$ in ~\eqref{Rekj0}, we get
												\begin{equation}\label{Jaco29}
													\sum_{j=1}^n \lambda_j^2=\left(\sum_{j=1}^n \lambda_j\right)^2-\sum_{j\neq k} (\lambda_j \lambda_k)=\left(\sum_{j=1}^n \Re\, z_j\right)^2-\sum_{j\neq k} \Re(z_j z_k).
												\end{equation}
												Dividing \eqref{rew10} by $p(z) = \prod_{j=1}^n (z-\lambda_j)$ we get
												\begin{equation}\label{Jaco30}
													1+il+ilz\sum_{j=1}^n \frac{w_j}{\lambda_j-z}=\frac{\prod_{k=1}^n (z-z_k)}{\prod_{k=1}^n (z-\lambda_k)}.
												\end{equation}
												Taking residues at $z=\lambda_j$ we get
												\begin{equation}\label{Jaco31}
													-il\lambda_j w_j=\frac{\prod_{k=1}^n (\lambda_j-z_k)}{\prod_{k\neq j} (\lambda_j-\lambda_k)}.
												\end{equation}
												Note that 
												\begin{equation}\label{Jaco32}
													\frac{1}{2}\prod_{j=1}^n (z-z_j)+\frac{1}{2}\prod_{j=1}^n (z-\bar{z_j})=\sum_{j=0}^{n}(\Re\, \kappa_j)z^j=\prod_{j=1}^n (z-\lambda_j).
												\end{equation}
												If we take $z=z_k$, $k=1,\ldots,n$ in \eqref{Jaco32}, substitute these equations into \eqref{Jaco31}, and use \eqref{magni2}, we get
												\begin{equation}\label{Jaco33}
													\prod_{j=1}^n w_j=	\frac{\prod_{j,k} |z_j-\bar{z_k}|}{(2l)^n \prod_{j<k}  |\lambda_j-\lambda_k|^{2}|\Re(\prod_{j=1}^n z_j)|}.
												\end{equation}
												Finally, combining \eqref{Jaco33}, \eqref{Jaco29}, \eqref{magni2} into \eqref{Jaco28}, we arrive at 
												\begin{align*}
													\frac{F(l) l^{1-\beta n/2}}{g_{\beta,n}c_{\beta,n} 2^{n(\beta/2-1)}} & e^{-\frac{1}{2}((\sum_{j=1}^n \Re\, z_j)^2-\sum_{j\neq k} \Re(z_j z_k))}\nonumber\\
													& \times  \frac{1}{|\Re(\prod_{j=1}^n z_j)|^{\beta/2}}\prod_{j,k=1}^n |z_j-\bar{z_k}|^{\beta/2-1}\prod_{1\leq j<k\leq n}^n |z_j-{z_k}|^2 d^2 z_1\cdots d^2 z_n,\label{hugeNEW}
												\end{align*}
												where $l$ is given by \eqref{magni1}. 
												This proves \eqref{huge1}.
											\end{proof}

\section{Laguerre $\beta$-ensembles}
\subsection{Laguerre $\beta$-ensemble and the spectral measure}
In the same paper~\cite{Dumede02}, Dumitriu--Edelman introduced the Laguerre $\beta$--ensembles $\mathrm{L}\beta \mathrm{E}_{(m,n)}$ as the tridiagonalization of the Laguerre orthogonal/unitary ensembles, subsequently\\ $\beta$-extrapolated (see also Section~\ref{ss:GOE} below). 

For $m\ge n$, an element of $\mathrm{L}\beta \mathrm{E}_{(m,n)}$ are defined as the $n\times n$ Jacobi matrix $\calJ= B^* B$, where
\begin{equation}\label{bidiag}
	B = \left(
	\begin{array}{ccccc}
		x_1&y_1&0& &\\
		0&x_2&y_2&\ddots &\\
		0&0&x_3&\ddots & 0 \\
		&\ddots&\ddots&\ddots & y_{n-1} \\
		& & 0 & 0 & x_n
	\end{array}\right),
\end{equation}
with
\begin{equation}\label{eq:LagCoeff}
\begin{aligned}
	x_j & \sim \chi_{\beta(m-j+1)},  & 1\le j\le n, \\
	y_j & \sim \chi_{\beta(n-j)},  & 1\le j\le n-1,
\end{aligned}
\end{equation}
where $\chi_k$ is the chi distribution with $k$ degrees of freedom, and $\beta>0$ is arbitrary. 

For $m \le  n-1$, we define (see~\cite{Koz17}) an element of $\mathrm{L}\beta \mathrm{E}_{(m,n)}$  as the $(m+1)\times (m+1)$ Jacobi matrix  $\calJ= B^* B$ of rank $m$, where
\begin{equation}\label{bidiag2}
	B = \left(
	\begin{array}{cccccc}
		x_1&y_1&0& &\\
		0&x_2&y_2&\ddots &\\
		0&0&x_3&\ddots & 0 \\
			&\ddots&\ddots&\ddots&y_{m-1}& 0\\
		&&&0&x_{m} & y_{m} \\
		&& &  & 0 & 0
	\end{array}\right),
\end{equation}
with
\begin{equation}\label{eq:LagCoeff2}
	\begin{aligned}
	x_j & \sim \chi_{\beta(m-j+1)},  & 1\le j\le m, \\
	y_j & \sim \chi_{\beta(n-j)},  & 1\le j\le m.
\end{aligned}
\end{equation}

Let us now describe their spectral measures~\eqref{eq:spWithJ}.  Let $a=|m-n|+1-2/\beta$.

For $m\ge n$, $\mu$ has the form~\eqref{spectralM}, ~\eqref{domain}, with the joint probability distribution~\cite{Dumede02}
\begin{equation}\label{eq:Laguerre1}
	\frac1{s_{\beta,n,m}}\prod_{j=1}^n \left(\lambda_j^{\beta a/2} e^{-\lambda_j/2}\right)\prod_{j<k}  |\lambda_j-\lambda_k|^\beta \prod_{j=1}^n  w_j^{\beta/2-1},
\end{equation}
where
\begin{equation}
	s_{\beta,n,m}= \frac{2^{n(a\beta/2+1+(n-1)\beta/2)}\Gamma(\beta/2)^n}{\Gamma(\beta n/2)}\prod_{j=1}^n \frac{\Gamma(1+\beta j/2)\Gamma(1+\beta a/2+ \beta(j-1)/2)}{\Gamma(1+\beta/2)}.
\end{equation}

For $m\le n-1$,  $\mu$ takes the form (\cite[Prop. 1]{Koz17})
\begin{align}\label{spectralMwith0}
	&\mu(x) = w_0 \delta_{0}+\sum_{j=1}^m w_j \delta_{\lambda_j} , 
	\\
	\label{domainwith0}
		& \sum_{j=0}^m w_j =1;  \quad w_j > 0, \quad 0 \le j \le m; \quad \lambda_j \in \bbR,
\end{align}
with the joint probability distribution
	\begin{align}\label{eq:LagJointHard}
	\frac{w_0^{\beta(n-m)/2-1}}{t_{\beta,n,m}}\prod_{j=1}^m \left(\lambda_j^{\beta a/2} e^{-\lambda_j/2}\right)\prod_{j<k}  |\lambda_j-\lambda_k|^\beta \prod_{j=1}^m  w_j^{\beta/2-1},
\end{align}
where
\begin{multline}
	t_{\beta,n,m}= \frac{2^{m(a\beta/2+1+(m-1)\beta/2)}\Gamma(\beta/2)^m}{\Gamma(\beta n/2)}\prod_{j=1}^m \frac{\Gamma(1+\beta j/2)\Gamma(1+\beta a/2+ \beta(j-1)/2)}{\Gamma(1+\beta/2)}\\
	\times \Gamma(\beta(n-m)/2).
\end{multline}

 Observe that, $\calJ\in\mathrm{L}\beta\mathrm{E}_{(m,n)}$ is a square matrix of order  $N=\min\{n,m+1\}$.

We want to compute the joint distribution of the eigenvalues of 
\begin{equation}
	\calJ_{l,\times}:=(I_N+ile_1e_1^*)\mathcal{J},
\end{equation}
where $\calJ$ is from $\mathrm{L}\beta\mathrm{E}_{(m,n)}$.

All possible eigenvalue configurations of $\calJ_{l,\times}$ are given by Lemma~\ref{thm:HB}(b) for the case $m \ge n$ and  by Lemma~\ref{thm:HB}(c) for the case $m \le n-1$. In the next two sections we compute their distribution.


\subsection{Eigenvalue density for multiplicative perturbations: $m\geq n$ case.}\label{ss:Laguerre1}

\begin{theorem}\label{thm:Lag1}
Suppose  $\beta>0$, $m\geq n$, $\calJ$ belongs to the Laguerre $\beta$-ensemble $\mathrm{L}\beta\mathrm{E}_{(m,n)}$, and  $l>0$ 
is independent of $\calJ$ with a given absolutely continuous distribution $d\nu(l)=F(l)dl$.
Then the eigenvalues $\{z_j\}_{j=1}^n$  of $\calJ_{l,\times}$ ~\eqref{genericJ2new} are distributed on 
\begin{equation}\label{eq:zerosLag1}
	\left\{ \{z_j\}_{j=1}^n \in (\bbC_+)^n: \tan\Big(\sum_{j=1}^n \Arg  z_j \Big) \in \supp\,\nu
	\right\}.
\end{equation}
according to
\begin{align}
	\frac{1}{C_{\beta,n,m}} &e^{-\sum_{j=1}^n \Re\, z_j/2} \Big\vert\Re\big(\prod_{j=1}^n z_j\big)\Big\vert^{a\beta/2-\beta/2}\prod_{j,k=1}^n |z_j-\bar{z}_k|^{\beta/2-1}\prod_{1\leq j<k\leq n} |z_j-{z_k}|^2\nonumber\\
	&\times \frac{F(l)}{l^{\tfrac{\beta n}{2}-1}} \,  d^2 z_1\ldots d^2 z_n,\label{huge11}
\end{align}
	where $l=\tan\Big(\sum_{j=1}^n \Arg z_j\Big)$
	and $C_{\beta,n,m} = s_{\beta,n,m} 2^{n(\beta/2-1)}$. 
\end{theorem}
\begin{proof}
	The proof for this case is identical to the proof of Theorem \ref{Gauss} with the only difference being the starting distribution~\eqref{eq:Laguerre1}.	 
	Indeed, applying the Jacobian~\eqref{Jaco1} to~\eqref{eq:Laguerre1}, and then using~\eqref{Rekj0}, \eqref{magni1}, \eqref{magni2}, \eqref{Jaco33} produces~\eqref{huge11}.
\end{proof}
\begin{remark}
For comparison, the joint distribution of rank-one {\it additive} perturbations $\calJ+ i l e_1 e_1^*$ is proportional to
\begin{equation*}
	e^{- \sum_{j=1}^n \Re \,z_j/2 } \Big\vert \Re\big(\prod_{j=1}^n z_j\big) \Big\vert^{a\beta /2} \prod_{j,k=1}^n |z_j-\bar{z}_k |^{\beta/2 -1} \prod_{1\le j<k\le n} |z_j-z_k|^2 
	  \frac{F(l)}{l^{\tfrac{\beta n}{2}-1}} d^2 z_1\ldots d^2 z_n,
	\end{equation*}
where $l=\sum_{j=1}^n \Im\, z_j$, see  ~\cite{Koz17} for the proof and more details.
These two distributions look astonishingly similar to each other, even more so than for the Gaussian $\beta$-ensembles, see Remark 1 after Theorem~\ref{Gauss}. This is very likely not a coincidence. One could envision that the elegant trick of Forrester~\cite[Sect~3.1]{ForrRev} could provide an explanation here at least for $\beta=1,2$. The idea there is that the spectrum of a multiplicative perturbation of Wishart matrix $X^* X \Omega$ can be related to the spectrum of $X^*\Omega X$. The latter matrix can be shown to have the structure of an {\it additive} perturbation of a Wishart matrix with one less column. This explains the shift by $1$ in the parameter $a$ for the two distributions above. 
Observe that this trick will not work for Gaussian matrices. 
\end{remark}

\subsection{Eigenvalue density for multiplicative perturbations: $m\leq n-1$ case.}\label{ss:Laguerre2}


First we need the analogue of Lemma~\ref{lm:Jacq} for measures with a mass point at $0$. 
\begin{lemma}\label{lm:Jacq2}
	Let $\mu$ be given by~\eqref{spectralMwith0},~\eqref{domainwith0}, and let $q(z)= \sum_{j=0}^{m} q_j z^j$ with $q_{m}=1$ be the characteristic polynomial of $\calJ^{(1)}$. Treating $\lambda_j$  as constants,
	\begin{equation}\label{eq:Jacq2}
		\left\lvert \det{\frac{\partial{(q_0,\ldots, q_{m-1})}}{\partial{(w_1,\ldots, w_{m})}}} \right\rvert= \prod_{j=1}^m|\lambda_j| \prod_{1\leq j<k\leq m}   |\lambda_j-\lambda_k|.
	\end{equation}
\end{lemma}
\begin{proof}
	This is~\cite[Eq.(5.36)]{Koz17} without the $l^{m}$ factor and without the factor of 
	$$
	\left\lvert \det\frac{\partial(\Re\,\kappa_1,\ldots,\Re\,\kappa_m)}{\partial(\lambda_1,\ldots,\lambda_m)}\right\rvert= \prod_{j<k}|\lambda_j-\lambda_k|.
	$$
	Indeed, in their setting $-\frac1l \sum_{j=0}^{m} (\Im\,\kappa_j) z^j$ is our $q(z)$. 
\end{proof}

Compared to the distribution in Theorems \ref{thm:Lag1}, the distribution~\eqref{huge111} picks up a factor with a power of $\lvert\Im\big(\prod_{j=1}^m z_j\big)-l\Re\big(\prod_{j=1}^m z_j\big)\rvert$ that replaces the factor with a power of $\lvert\Re\big(\prod_{j=1}^m z_j\big)\rvert$ in~\eqref{huge11}.
\begin{theorem}\label{laguerredifficult}
	Suppose $\beta>0$,  $m\leq n-1$, $\calJ$ belongs to the Laguerre $\beta$-ensemble $\mathrm{L}\beta\mathrm{E}_{(m,n)}$, and  $l>0$ is fixed. 
	Then $\calJ_{l,\times}$ has one zero eigenvalue and the remaining $m$ eigenvalues 
	are distributed on
	\begin{equation}
		\left\{ \{z_j\}_{j=1}^{m} \in (\bbC_+)^{m}: \sum_{j=1}^{m} \Arg\, z_j < \arctan l
		\right\}
	\end{equation}
	according to
	\begin{align}
		\frac{l^{1-\beta n/2}}{C_{\beta,m,n}} & e^{-\sum_{j=1}^m \Re\, z_j/2}
		 \left\lvert\Im\Big(\prod_{j=1}^m z_j\Big)-l\Re\Big(\prod_{j=1}^m z_j\Big)\right\rvert^{\beta(n-m)/2-1}
		 \nonumber\\
		&\times \prod_{j,k=1}^m |z_j-\bar{z_k}|^{\beta/2-1}\prod_{1\leq j<k\leq m} |z_j-{z_k}|^2\,
		   d^2 z_1\ldots d^2 z_m.\label{huge111}
	\end{align}
where $C_{\beta,m,n}=t_{\beta,n,m}2^{m(\beta/2-1)}$.
\end{theorem}
\begin{proof}
	By Lemma~\ref{thm:HB}(c), $\calJ_{l,\times}$ has a zero eigenvalue, so we let
	\begin{equation}\label{eq:charLagHard}
		\det(z I_{m+1}-\calJ_{l,\times})=z\prod_{j=1}^m (z-z_j)= z\sum_{j=0}^m \kappa_j z^j
	\end{equation}
	with $\kappa_m=1$.
	The Jacobian from $z_j$'s to $\kappa_j$'s is known (see~\eqref{vande0}), so we are left with computing the Jacobian of the transformation from $\{\kappa_j\}_{j=0}^{m-1}$ to $\{\lambda_j\}_{j=1}^m$, $\{w_j\}_{j=1}^m$, in~\eqref{spectralMwith0},~\eqref{domainwith0}.
	
	Using \eqref{eq:HB} and dividing by $z$, we get:
	\begin{equation}\label{rew1}
		\sum_{j=0}^m \kappa_j z^j=\prod_{j=1}^m (z-\lambda_j)+il\prod_{j=1}^m (z-\lambda_j)-ilq(z),
	\end{equation}
	where $q(z)$ is the characteristic polynomial of $\calJ^{(1)}$. 
	
	It follows that
	\begin{align}
		& \label{Rekj}
		\sum_{j=0}^{m}(\Re\, \kappa_j)z^j=\prod_{j=1}^m (z-\lambda_j), \\
		& \label{rew2}
		\sum_{j=0}^{m-1} (\Im \,\kappa_j) z^j=l\prod_{j=1}^m (z-\lambda_j)-lq(z).
	\end{align}
	Then~\eqref{Rekj} implies 
	\begin{align}
		\label{Jacob11}
		&\left\lvert \det{\frac{\partial{(\Re\,\kappa_0,\ldots, \Re\,\kappa_{m-1})}}{\partial{(\lambda_1,\ldots, \lambda_{m})}}} \right\rvert= \prod_{1\leq j<k\leq m}  |\lambda_j-\lambda_k|,
		\\
		\label{0mat1}
		&{\frac{\partial{(\Re\,\kappa_0,\ldots, \Re\,\kappa_{m-1})}}{\partial{(w_1,\ldots, w_{m})} }}=[0]_{m\times m}.
	\end{align}

	Treating $l$ and $\lambda_j$'s as constant and using~\eqref{rew2} and Lemma~\ref{lm:Jacq2} we obtain
	\begin{equation}\label{kap31}
	\left\lvert \det{\frac{\partial{(\Im\,\kappa_0,\ldots, \Im\,\kappa_{m-1})}}{\partial{(w_1,\ldots, w_{m})}}} \right\rvert=l^{m} \prod_{j=1}^m\lambda_j \prod_{1\leq j<k\leq m}  |\lambda_j-\lambda_k|. 
	\end{equation}
	Just like in the proof of Lemma~\ref{lem:Jacobian}, using the block structure and the presence of the zero block~\eqref{0mat1} we now can compute the Jacobian 
	\begin{align}\label{kap61}
	\left\lvert \det{\frac{\partial{(\Re\, z_1,\ldots,\Re\, z_{m},\Im\, z_{1},\ldots,\Im\, z_{m})}}{\partial{(\lambda_1,\ldots,\lambda_m,w_1,\ldots,w_m)}}}\right\rvert= \frac{l^{m}\prod_{1\leq j<k\leq m}  |\lambda_j-\lambda_k|^2 \left\lvert\prod_{j=1}^m \lambda_j\right\rvert}{ \prod_{1\leq j<k\leq m}|z_j-z_k|^2}.
\end{align}	

Applying the change of variables~\eqref{kap61} to the joint density~\eqref{eq:LagJointHard} we obtain the joint density of $z_j$'s:
\begin{equation}\label{longbigeq}
	\frac{w_0^{\beta(n-m)/2-1}}{l^m t_{\beta,n,m}}\prod_{j=1}^m \left(\lambda_j^{\beta a/2-1} e^{-\lambda_j/2}\right)  \prod_{j<k}  |\lambda_j-\lambda_k|^{\beta-2} \prod_{j=1}^m  w_j^{\beta/2-1} \prod_{1\leq j<k\leq m}|z_j-z_k|^2.
\end{equation}
To complete the proof we need to  express \eqref{longbigeq} in terms of the $z_j$'s.
By checking the $z^0$ and $z^{m-1}$ term of~\eqref{Rekj}, see also~\eqref{eq:charLagHard}, we get 
\begin{align}
& \prod_{j=1}^m \lambda_j=\Re\prod_{j=1}^m z_j\label{theeq1},\\
& \prod_{j=1}^m e^{ -\lambda_j/2}
=e^{-\sum_{j=1}^m \Re\, z_j/2}.\label{theeq2}
\end{align}
Denote $m(z) = -\frac{q(z)}{p(z)}$, where $p(z) = z\prod_{j=1}^m (z-\lambda_j)$. Then \eqref{eq:charLagHard} and~\eqref{rew1} imply 
\begin{equation}\label{kap81}
\frac{\prod_{j=1}^m (z-z_j)}{\prod_{j=1}^m (z-\lambda_j)}=1+il+ilz{m}(z).
\end{equation}
Because of 
\eqref{spectralMwith0},
 ${m}(z)= \frac{-w_0}{z}+ \sum_{j=1}^m \frac{w_j}{\lambda_j-z}$.  Taking the limit of both sides in \eqref{kap81} as $z\rightarrow 0$, we get
\begin{equation}\label{kap91}
	\frac{\prod_{j=1}^m z_j}{\prod_{j=1}^m \lambda_j}=1+il-ilw_0.
\end{equation}
Substituting \eqref{theeq1} into \eqref{kap91}, we get
\begin{equation}\label{theeq3}
	w_0=\left\lvert \frac{\Im\prod_{j=1}^m z_j-l\Re\prod_{j=1}^m z_j }{l\Re\prod_{j=1}^m z_j} \right\rvert.
\end{equation}
Similarly, for $j\neq 0$, taking residues at $\lambda_j$ in \eqref{kap81} implies
\begin{equation}\label{kap101}
	w_j=\left\lvert \frac{\prod_{k=1}^m (\lambda_j-z_k)}{l\lambda_j\prod_{k\neq j} (\lambda_j-\lambda_k)} \right\rvert.
\end{equation}
Using the equation
\begin{equation}\label{kap111}
	\frac{1}{2}\prod_{j=1}^m(z-z_j)+	\frac{1}{2}\prod_{j=1}^m(z-\bar{z}_j)=\sum_{j=0}^m (\Re\, \kappa_j)z^j=\prod_{j=1}^m (z-\lambda_j)
\end{equation}
for $z=z_k$, $k=1,\ldots,m$, and then combining it with \eqref{kap101}
we get
\begin{equation}\label{theeq4}
	\prod_{j=1}^m w_j=\left\lvert\frac{\prod_{j,k}(\lambda_j-z_k)}{l^m \prod_{j=1}^m \lambda_j \prod_{j<k}|\lambda_j-\lambda_k|^2}\right\rvert=\left\lvert\frac{\prod_{j,k}(\bar{z_j}-z_k)}{(2l)^m \prod_{j=1}^m \lambda_j \prod_{j<k}|\lambda_j-\lambda_k|^2}\right\rvert
\end{equation}
Substituting \eqref{theeq1},\eqref{theeq2}, \eqref{theeq3}, \eqref{theeq4} into \eqref{longbigeq} we 
obtain \eqref{huge111} which completes the proof.

\end{proof}
	

\section{Multiplicative perturbations of the classical ensembles}

\subsection{Multiplicative perturbations of $\mathrm{GOE}_n$ and $\mathrm{GUE}_n$}
\label{ss:GOE}

Let $X_{m\times n}$ be an $m\times n$ matrix with entries being i.i.d. real ($\beta=1$) or complex ($\beta=2$) centered normal random variables with  $\mathbb{E}(|X_{ls}|^2) = \beta$ for all $l,s$.	
Then we say that the $n\times n$ random  Hermitian matrix
$H=\tfrac{1}{2} (X_{n\times n} + X_{n\times n}^*)$ belongs to the Gaussian orthogonal ($\mathrm{GOE}_n$) or Gaussian unitary ensemble ($\mathrm{GUE}_n$), respectively. 

A (real) vector $v$ is called cyclic on an $n$-dimensional vector space $V$ with respect to a matrix $A$ if $v, Av,\ldots, A^{n-1}v$ is a basis for $V$. 

Assume that the first standard column $n$-vector $e_1$ is cyclic for $H$ (which is a probability 1 event). Then  Dumitriu--Edelman~\cite{Dumede02} (see also Trotter's~\cite{Tro}) showed that $H$ from $\mathrm{GOE}_n$ and $\mathrm{GUE}_n$ 
can be reduced to the tridiagonal form ~\eqref{genericJwithn} with~\eqref{eq:GbetaCoefficients}
via $H  = S^*\calJ S$ with an orthogonal/unitary  (for $\beta =1$ or $2$, respectively) matrix $S$  satisfying 
\begin{equation}\label{S-equ}
	Se_1=S^* e_1=e_1.
\end{equation}

Let us now consider the multiplicative perturbation $H_{l,\times}$ of $H$ as in~\eqref{eq:mult}. Here   $\Gamma$ is an $n\times n$ positive definite matrix of rank 1 which is 
independent of $H$
 with real (if $\beta=1$) or complex (if $\beta=2$)
entries. 
Let $l=||\Gamma||_{HS}$ be the Hilbert--Schmidt norm of $\Gamma$. Then there is an orthogonal/unitary matrix $U$ (for $\beta =1,2$, respectively) such that $\Gamma=U(l e_1 e_1^*) U^*$. 
Then
\begin{equation}\label{eq:multiplicativePert}
	H_{l,\times}:= (I_n+i\Gamma)H = U(I_n+il e_1 e_1^*)U^*H .
\end{equation}
By the orthogonal/unitary invariance of $\mathrm{GOE}$/$\mathrm{GUE}$, the matrix $U^*H U$ belongs to the same ensemble as $H$. 
By the Dumitriu--Edelman tridiagonalization method described above, we can find a unitary $S$ and a tridiagonal matrix $\mathcal{J}$~\eqref{genericJwithn} such that $U^*HU= S^*\mathcal{J} S$. Using this along with \eqref{S-equ} we obtain
\begin{align}
	\label{eq:argument1}
	SU^* H_{l,\times} US^*&= S(I_n+il e_1 e_1^*)U^*H US^*\\
	\label{eq:argument2}
	&=S(I_n+il e_1 e_1^*)S^*\mathcal{J}\\
	& =\mathcal{J}+ il  e_1 e_1^*  \mathcal{J}, 
	\label{J-rep}
\end{align}
where $\mathcal{J}$ is from the Gaussian $\beta$-ensemble with $\beta=1,2$, respectively. This shows that the distribution of the eigenvalues of multiplicative perturbations ~\eqref{eq:multiplicativePert} for $\mathrm{GOE}_n$/$\mathrm{GUE}_n$ coincides with the distribution found in Theorem~\ref{Gauss} with $\beta=1$ and $2$, respectively, with $l=||\Gamma||_{HS}$.

\subsection{Multiplicative perturbations of $\mathrm{LOE}_{(m,n)}$ and $\mathrm{LUE}_{(m,n)}$}
\label{ss:LOE}
The same argument works for the Wishart orthogonal and unitary ensembles which we discuss now.
Given $X_{m\times n}$ as in Section~\ref{ss:GOE}, 
we say that the $n\times n$ random Hermitian matrix $X_{m\times n}^* X_{m\times n}$ belongs to the Wishart orthogonal/unitary 
($\mathrm{LOE}_{(m,n)}$, $\mathrm{LUE}_{(m,n)}$), respectively. 

If $m\ge n$ then~\cite{Dumede02} the tridiagonalization of $\mathrm{LOE}_{(m,n)}$, $\mathrm{LUE}_{(m,n)}$ is the $n\times n$ matrix ~\eqref{bidiag} with~\eqref{eq:LagCoeff} with $\beta=1,2$, respectively. Then the same argument as for $\mathrm{GOE}_n$/$\mathrm{GUE}_n$ above shows that the distribution of the eigenvalues of multiplicative perturbations
\begin{equation}\label{eq:multiplicativePertWishart}
	H_{l,\times}:= (I_n+i\Gamma)X^* X 
\end{equation}
of Wishart matrices $X^* X$ from $\mathrm{LOE}_n$/$\mathrm{LUE}_n$ coincides with the distribution found in Theorem~\ref{thm:Lag1} with $\beta=1$ and $2$, respectively, with $l=||\Gamma||_{HS}$.

If $m\le n-1$ then the tridiagonalization of the $n\times n$ matrix $\mathrm{LOE}_{(m,n)}$, $\mathrm{LUE}_{(m,n)}$ has the following form: it has block diagonal structure with the first $m\times m$ block being ~\eqref{bidiag2} with~\eqref{eq:LagCoeff2}, and the second block being $(n-m-1)\times(n-m-1)$ zero matrix, see~\cite[Lem 2(ii)]{Koz17}.
This is due to the fact that $e_1$ is not cyclic for $H$ in this case, so the spectral measure of $\calJ$ is supported on $m+1$ points (one of which is 0). Effectively, the spectral measure $\mu$ ignores the remaining $n-m-1$ zero eigenvalues of  $H$ that comes from the final $(n-m-1)\times(n-m-1)$ zero block of the  tridiagonalization of $H$.

As a result,  the spectrum of multiplicative perturbations ~\eqref{eq:multiplicativePertWishart} for $X^* X$ from $\mathrm{LOE}_{(m,n)}$/$\mathrm{LUE}_{(m,n)}$ for the case $m\le n-1$ consists of $n-m$ zero eigenvalues ($n-m-1$ coming from the zero block and $1$ from Theorem~\ref{laguerredifficult}) and, additionally, random points distributed according to ~\eqref{huge111}.


\subsection{Multiplicative perturbations of $\mathrm{GSE}_n$ and  $\mathrm{LSE}_{(m,n)}$}\label{ss:GSE}

Let us also briefly discuss Gaussian and Laguerre symplectic ensembles $\mathrm{GSE}_n$ and $\mathrm{LSE}_{(m,n)}$ 
($\beta=4$) and the difference from the $\beta=1,2$ cases above. 

Denote the algebra of real quaternions by $\bbH$. For $\bm{q}\in\bbH$, $\bm{q}=q_1+q_2 \bm{i} +  q_3 \bm{j}+  q_4\bm{k}$ (with $q_j\in\bbR$)  we let 
\begin{align*}
	\bar{\bm{q}} &:=q_1-q_2 \bm{i}-  q_3 \bm{j}- q_4\bm{k}, \\
	|\bm{q}| &:=\sqrt{\bm{q}\bar{\bm{q}}}=\sqrt{q_1^2+q_2^2+q_3^2+q_4^2}.
\end{align*}
For a matrix $\bm{A}=(\bm{a}_{ls})$  of real quaternions we define the conjugate transposed matrix $\bm{A}^*:= (\bar{\bm{a}}_{sl})$ as usual. We call $\bm{U}$ unitary if $\bm{U} \bm{U}^* = \bm{U}^*\bm{U} = \bm{I}_n$. 

Let $\bm{X}_{m\times n}$ be an $m\times n$ matrix whose entries are i.i.d. normal real quaternions $X_{ls}\in\bbH$ with $\mathbb{E}(|X_{ls}|^2) = 4$ for all $1\le l\le m$, $1\le s\le n$. Then an $n\times n$ quaternionic matrix $\tfrac{1}{2} (\bm{X}_{n\times n} + \bm{X}_{n\times n}^*)$ is said to belong to the Gaussian symplectic ensemble $\mathrm{GSE}_n$ and  $\bm{X}_{m\times n}^* \bm{X}_{m\times n}$ is said to belong to the Wishart symplectic ensemble $\mathrm{LSE}_{(m,n)}$, cf. the construction for orthogonal/unitary ensembles above.

The triadiagonalization of $\bm{H}$ from $\mathrm{GSE}_n$ and $\mathrm{LSE}_{(m,n)}$ 
takes the form
$\bm{H}  = \bm{S}^*\bm{\calJ} \bm{S}$ with a unitary matrix $\bm{S}$  satisfying 
\begin{equation}\label{S-equ2}
	\bm{S}\bm{e}_1=\bm{S}^* \bm{e}_1=\bm{e}_1
\end{equation}
(here $\bm{e}_1$ is the $n$-dimensional vector of real quaternions with $1$ as the first entry and $0$ everywhere else). The resulting Jacobi matrix $\bm{\calJ}$ has the exact same form as the  Gaussian or Laguerre 
$\beta$-ensemble, respectively, with $\beta=4$, with the only wrinkle being that each of the Jacobi coefficients $\bm{a}_j$'s and $\bm{b}_j$'s are still quaternions that happen to be purely real (that is, their $\bm{i}$-, $\bm{j}$-, $\bm{k}$-components are zero). 


Now we consider the multiplicative perturbation
\begin{equation}\label{eq:multiplicativePertGSE}
	\bm{H}_{l,\times}:= (\bm{I}_n+\bm{\Gamma})\bm{H},
\end{equation}
where $\bm\Gamma = -\bm{\Gamma}^*$ is skew-Hermitian and of rank $1$ (see~\cite{Rod}). Such matrices can be diagonalized as
\begin{equation}\label{eq:quatDiagonalization}
	i\Gamma = \bm{U}(l \bm{i} \bm{e}_1 \bm{e}_1^*) \bm{U}^*,
\end{equation}
with $\bm{U}$ unitary, see~\cite[Thm 5.3.6(d)]{Rod} for the proof, 
with $l = \big(\sum_{j,k=1}^n |\bm{\Gamma}_{jk}|^2\big)^{1/2}$.
Note that in~\eqref{eq:quatDiagonalization}, $\bm{i} \bm{e}_1 \bm{e}_1^*$ is just the $n\times n$ quaternionic matrix with $\bm{i}$ in the $(1,1)$-entry and $\bm{0}$ everywhere else.

The remaining arguments in~\eqref{eq:multiplicativePert} and~\eqref{eq:argument1}--\eqref{J-rep} goes through without changes and lead us to studying multiplicative perturbations of Jacobi matrices
\begin{equation}\label{eq:quaternionJacobiMult}
	\bm{\calJ}_{l,\times}=(\bm{I}_n+l\bm{i} \bm{e}_1 \bm{e}_1^*)\bm{\mathcal{J}}.
\end{equation} 

Finally, to study the eigenvalues of these matrices, it is convenient to represent each  quaternion $\bm{q}=q_1+q_2 \bm{i} +  q_3 \bm{j}+  q_4\bm{k}$  as the $2\times 2$ complex matrix 
\begin{equation}\label{eq:quaternion}
	\begin{pmatrix}
		q_1+iq_2 & q_3+iq_4 \\
		-q_3+iq_4 & q_1-iq_2
	\end{pmatrix},
\end{equation}
and any $n\times n$ quaternionic matrix as $2n\times 2n$ complex matrix with the corresponding block structure. 

Now observe that in~\eqref{eq:quaternionJacobiMult} matrices $\bm{I}_n$ and $\bm{\calJ}$ are then represented as the $2n\times 2n$ complex matrices with each $2\times 2$ block being a multiple of $I_2$. The matrix $l\bm{i} \bm{e}_1 \bm{e}_1^*$ is represented as the $2n\times 2n$ complex matrix with the block 
\begin{equation*}
	\begin{pmatrix}
		li & 0\\
		0 & -li
	\end{pmatrix},
\end{equation*}
in the top-left corner and $0$'s everywhere else. We can conclude then that the complex-matrix representation of $\bm{\calJ}_{l,\times}$ splits into the direct sum of two $n\times n$ matrices $({I}_n+l{i} {e}_1 {e}_1^*){\mathcal{J}}$ and $({I}_n-l{i} {e}_1 {e}_1^*){\mathcal{J}}$, where $\calJ$ is the Jacobi matrix from the corresponding (scalar) $\beta$-ensemble with $\beta =4$. The resulting spectrum is therefore composed of the same eigenvalues as in Theorems~\ref{Gauss}, \ref{thm:Lag1}, and~\ref{laguerredifficult}, respectively, but combined  with their complex conjugated images. 

We remark that the same ``eigenvalue splitting'' effect happens for the additive perturbations of Gaussian, Wishart, and chiral Gaussian symplectic ensembles. This was overlooked in our~\cite{Koz17} and \cite{AlpRos}.

\subsection{Multiplicative perturbations of Chiral Gaussian ensembles $\mathrm{chGOE}_{(m,n)}$, $\mathrm{chGUE}_{(m,n)}$, $\mathrm{chGSE}_{(m,n)}$}\label{ss:chiral}
Let  $X_{m\times n}$ be again as in Section~\ref{ss:GOE} with $\beta=1, 2$, respectively, and as in Section~\ref{ss:GSE} if $\beta=4$. We say that the $(m+n)\times(m+n)$ random  Hermitian matrix 
\begin{equation}\label{eq:chiral}
	{H} = \begin{pmatrix}
		\textbf{0}_{n\times n} & X_{m\times n}^* \\
		X_{m\times n}& \textbf{0}_{m\times m}
	\end{pmatrix}
\end{equation}
belongs to the chiral Gaussian  orthogonal ($\beta=1$),  unitary $(\beta=2$), or symplectic ($\beta=4$) ensemble 
($\mathrm{chGOE}_{(m,n)}$, $\mathrm{chGUE}_{(m,n)}$, $\mathrm{chGSE}_{(m,n)}$),
respectively. Such matrices, which can be viewed as a random Dirac operator, are natural from the point of view of the quantum chromodynamics. 
A basic algebra (see, e.g.,~\cite[Prop 3.1.1]{Forrester-book}) tells us that non-zero eigenvalues of the chiral ensembles coincide with the $\pm$ square roots of the eigenvalues of the corresponding Laguerre ensembles.

Just as for the Gaussian and Laguerre ensembles, one can reduce these ensembles to the Jacobi form $\tilde\calJ$, see~\cite{AlpRos,dum2,jac} and compute the spectral measure~\cite[Sect~4]{AlpRos}. 

It is natural to consider multiplicative perturbations
\begin{equation*}
	H_{\times} = (I_{n+m}+\Omega)H,
\end{equation*}
where $\Omega=\Omega^*$ an $(n+m)\times(n+m)$ matrix of rank 1 with real or complex
for $\beta=1,2$, respectively,  with the block structure 
\begin{equation}
	{\Omega}=
	\begin{pmatrix}
		i\Gamma & \textbf{0}_{n\times m}\\\
		\textbf{0}_{m\times n} &\textbf{0}_{m\times m}
	\end{pmatrix}
\end{equation}
(for $\beta=4$ case we replace $i\Gamma$ with quaternionic skew-Hermitian matrix of rank $1$, in a similar fashion as in Section~\ref{ss:GSE}). 
Multiplying this out, we obtain
\begin{equation}\label{eq:multChiral}
	H_{\times} = 
	 \begin{pmatrix}
		\textbf{0}_{n\times n} & (I_n+i\Gamma) X_{m\times n}^* \\
		X_{m\times n}& \textbf{0}_{m\times m}
	\end{pmatrix}.
\end{equation}
Using the Schur complement formula one easily obtains
\begin{align}
	\det(\lambda I_{n+m} - H_\times) & = \det(\lambda I_m) \det(\lambda I_n - \lambda^{-1} (I_n+i\Gamma) X_{m\times n}^* X_{m\times n}) 
	\\
	& =
	\lambda^{m-n} \det(\lambda^2 I_n - (I_n+i\Gamma) X_{m\times n}^* X_{m\times n}).
\end{align}
This shows that nonzero eigenvalues of $H_\times$ are precisely $\pm$ square roots of the eigenvalues of the multiplicative perturbations~\eqref{eq:mult} of the Laguerre ensembles which are computed in Theorems~\ref{thm:Lag1} and ~\ref{laguerredifficult}.

\smallskip

One may argue that instead of~\eqref{eq:multChiral} a more natural model for the non-Hermitian multiplicative perturbations of chiral ensembles would be
\begin{multline*}
\begin{pmatrix}
	I_n + i\Gamma & \textbf{0}_{n\times m}\\\
	\textbf{0}_{m\times n} & I_m
\end{pmatrix}
\begin{pmatrix}
	\textbf{0}_{n\times n} & X_{m\times n}^* \\
	X_{m\times n}& \textbf{0}_{m\times m}
\end{pmatrix}
\begin{pmatrix}
	I_n  + i\Gamma & \textbf{0}_{n\times m}\\\
	\textbf{0}_{m\times n} & I_m
\end{pmatrix}
\\=
\begin{pmatrix}
	\textbf{0}_{n\times n} & (I_n+i\Gamma) X_{m\times n}^* \\
	X_{m\times n} (I_n+i\Gamma) & \textbf{0}_{m\times m}
\end{pmatrix}
\end{multline*}
Indeed, as above this can be related to the eigenvalues of 
\begin{equation}\label{eq:last}
(I_n+i\Gamma)  X_{m\times n}^* X_{m\times n} (I_n+i\Gamma)
\end{equation}
which has a strong resemblance with the spiked Wishart models.

However, since for square matrices $A$ and $B$ the spectra of $AB$ and $BA$ coincide, the spectrum of  perturbations~\eqref{eq:last} can be reduced to the spectrum of $(I_n+i\Gamma)^2  X_{m\times n}^* X_{m\times n}$ which is effectively our model~\eqref{eq:mult}.


\begin{thebibliography}{99}
		
		\bibitem{AlpRos} G.\ Alpan, R.\ Kozhan: \emph{Hermitian and non-Hermitian perturbations of chiral Gaussian $\beta$-ensembles}, J. Math. Phys., \textbf{63}, 043505 (2022)
		
		\bibitem{And} G. \ W. \ Anderson: \emph{A short proof of Selberg's generalized beta formula}, Forum Math. \textbf{3}, no.4, 415--417 (1991)
		
		\bibitem{Dix} A. \ L. \ Dixon: \emph{Generalization of Legendre's Formula Formula}, Proc. London Math. Soc. (2) \textbf{3}, 206--224 (1905)
		
		
		\bibitem{DubLas} G.\ Dubach, L. Erd\H{o}s: \emph{Dynamics of a rank-one perturbation of a Hermitian matrix}, Electron. Commun. Probab., \textbf{28}, 1--13 (2023)
		
		
			
			
			
		
		
		
		
		
		
		
		
				
		
		
		
		
		\bibitem{Dumede02} I.\ Dumitriu, A.\ Edelman: \emph{Matrix models for beta ensembles}, J. Math. Phys, \textbf{43}(11), 5380--5847 (2002)
		
		\bibitem{dum2}I.\ Dumitriu, P.\ J.\ Forrester: \emph{Tridiagonal realization of the antisymmetric Gaussian $\beta$-ensemble}, J. Math. Phys., \textbf{51}, 093302  (2010)
		
		\bibitem{Forrester-book} P.\ J.\ Forrester: \emph{Log-gases and random matrices}, Princeton University Press, Princeton, NJ, (2010)
		
		
		
		\bibitem{ForrRev}  P.\ J.\ Forrester: \emph{Rank 1 Perturbations in Random Matrix Theory - A Review of Exact Results}, Random Matrices: Theory Appl., \textbf{12}, No. 04, 2330001 (2023)
		
		\bibitem{fyo16}Y.\ V.\ Fyodorov: \emph{Random matrix theory of resonances: An overview}, in 2016 URSI International Symposium on Electromagnetic Theory (EMTS, 2016)
		
		
		
		\bibitem{fyokho99} Y.\ V.\ Fyodorov,  B.\ A.\ Khoruzhenko: \emph{Systematic analytical approach to correlation functions of resonances in quantum chaotic scattering}, Phys. Rev. Lett., \textbf{83} (1), 65--68 (1999)
		
		\bibitem{FKS}  Y.\ V.\ Fyodorov, B.\ A.\ Khoruzhenko, H.-J. Sommers: 		\emph{Almost-Hermitian random matrices: eigenvalue density in the complex plane},Phys. Lett. A \textbf{226}, no.1--2, 46--52 (1997)
		
		\bibitem{FKS2}  Y.\ V.\ Fyodorov, B.\ A.\ Khoruzhenko, H.-J. Sommers: \emph{Almost Hermitian random matrices: crossover from Wigner--Dyson to Ginibre eigenvalue statistics}, Phys. Rev. Lett. \textbf{79}, no.4, 557--560 (1997)
	
		\bibitem{fyosav15} Y.\ V.\ Fyodorov, D.\ Savin: \emph{Resonance Scattering of Waves in Chaotic Systems}, in The Oxford Handbook of Random Matrix Theory, 702--722. Oxford University Press, (2015)
		
		\bibitem{FyoSom96}  Y.\ V.\ Fyodorov, H.-J.\ Sommers: \emph{Statistics of S-matrix poles in few-channel chaotic scattering: crossover
		from isolated to overlapping resonances}, JETP Lett. \textbf{63} (12), 1026--1030 (1996)
		
		\bibitem{FyoSom97}  Y.\ V.\ Fyodorov, H.-J.\ Sommers: \emph{Statistics of resonance poles, phase shifts and time delays in quantum chaotic scattering: random matrix approach for systems with broken time-reversal invariance}, J. Math. Phys. \textbf{38}, no.4, 1918--1981 (1997)
		
		\bibitem{fyosom03}  Y.\ V.\ Fyodorov, H.-J.\ Sommers: \emph{Random Matrices Close to Hermitian or Unitary: Overview of Methods and Results}, J Phys A, \textbf{36}, 3303--3347 (2003)
		
		
		
		\bibitem{fyosomtit}  Y.\ V.\ Fyodorov, H.-J.\ Sommers, Y. \ V. \ Titov: 
		\emph{S-matrix poles for chaotic quantum systems as eigenvalues of
		complex symmetric random matrices: from isolated to overlapping resonances}, J. Phys. A \textbf{32} (5), L77--L87 (1999)
		
		\bibitem{jac}S.\ Jacquot, B.\ Valk\'{o}: \emph{Bulk scaling limit of the Laguerre ensemble}, Electron. J. Probab. \textbf{16}, 314--346 (2011)
		
		
		
		\bibitem{KK} R.\ Killip, R.\ Kozhan: \emph{Matrix models and eigenvalue statistics for truncations of classical unitary ensembles of random matrices}, Comm. Math. Phys., \textbf{349}, 991--1027 (2017) 
		
		\bibitem{Koz17} R.\ Kozhan: \emph{Rank one non-hermitian perturbations of hermitian $\beta$-ensembles of random matrices}. J. Stat. Phys. \textbf{168}, 92--108 (2017)
		
		\bibitem{Koz20} R.\ Kozhan: \emph{On Gaussian random matrices coupled to the discrete Laplacian}. in Oper. Theory Adv. Appl. (2020), Issue ``Analysis as a Tool in Mathematical Physics'', in memory of Boris Pavlov (eds P.Kurasov, A.Laptev, S.Naboko, and B.Simon)
		
		\bibitem{KT} R. Kozhan, M. Tyaglov: \emph{A Generalized Hermite--Biehler Theorem and Non-Hermitian Perturbations of Jacobi Matrices}, J. Math. Anal. Appl. 536, no.2, Paper No. 128241, 18 pp. (2024)
		
		\bibitem{nucl} G. E. Mitchell, A. Richter, H. A. Weidenmüller: \emph{Random Matrices and Chaos in Nuclear Physics: Nuclear Reactions},  Rev. Mod. Phys, \textbf{82}(4), 2845--2901 (2010)
		
		\bibitem{OroWoo1} S. O'Rourke and P. M. Wood: \emph{Spectra of nearly Hermitian random matrices}, Ann. Inst. Henri Poincaré Probab. Stat. \textbf{53}, no.3, 1241--1279 (2017)

				
		
		
		
		\bibitem{Rod} L. \ Rodman: \emph{Topics in quaternion linear algebra},
		Princeton University Press, Princeton, (2014)
		
		\bibitem{Sim11} B.\ Simon: \emph{Szeg\H{o}’s Theorem and Its Descendants: Spectral Theory for $L^2$ Perturbations of Orthogonal Polynomials}, Princeton University Press, Princeton, (2011)
		
		\bibitem{SokZel} V.\ V.\ Sokolov, V.\ G.\ Zelevinsky: \emph{Dynamics and statistics of unstable quantum states}, Nucl. Phys. A, \textbf{504} (3), 562--588 (1989)
		
		\bibitem{StoSeb} H.-J. St\"{o}ckmann, P.\ \v{S}eba: \emph{The joint energy distribution function for the Hamiltonian $H = H_0-iWW^+$ for the one-channel case}, J. Phys. A, \textbf{31} (15), 3439--3448 (1998)
		
		\bibitem{Tro} H. \ F. \ Trotter: \emph{Eigenvalue distributions of large Hermitian matrices; Wigner's semicircle law and a theorem of Kac, Murdock, and Szeg\H{o}}, Advances in Mathematics \textbf{54}, no. 1, 67--82 (1984)
		
		\bibitem{Ull} N.\ Ullah: \emph{On a generalized distribution of the poles of the unitary collision matrix}. J. Math. Phys., \textbf{10}, 2099--2103 (1969)
		
		
		
		
		\bibitem{Wen} B. Wendroff: \emph{On orthogonal polynomials}, Proc. Amer. Math. Soc. \textbf{12} 554--555 (1961).
		
		
		
		
		
	
		
	\end{thebibliography}
\end{document}